\documentclass[a4paper,12pt]{article}
\usepackage[utf8]{inputenc}
\usepackage[english]{babel}
\usepackage{amsthm}
\usepackage{amssymb} 
\usepackage{amsmath}
\usepackage{mathrsfs}
\usepackage{amssymb}
\usepackage{pdfsync}
\usepackage{graphicx}

\newtheorem{thm}{Theorem}[section]
\newtheorem{corollary}{Corollary}[section]
\newtheorem{definition}{Definition}
\newtheorem{propr}{Property}[section]

\newtheorem{lemma}{Lemma}
\newtheorem{rmk}{Remark}
\title{Spectral homogenization \\for a Robin-Neumann problem.}
\author{Andrea Cancedda\\ Dipartimento di Scienze Matematiche ``G. L. Lagrange",\\ Politecnico di Torino \\  corso Duca degli Abruzzi 24, 10129 Torino, Italy \\\texttt{andrea.cancedda@polito.it}}
\date{}                                           

\usepackage[autostyle]{csquotes}
\usepackage[hyperref,style=numeric,backend=biber,maxnames=4]{biblatex}
\usepackage{hyperref}

\def\e{\varepsilon}

\renewcommand{\Rn}[1]{\mathbb{R}^{#1}}
\newcommand{\eps}{\varepsilon}
\newcommand{\gammalim}[1]{\Gamma  \hbox{-} \lim_{\eps \to 0}\ {#1}}
\newcommand{\Omeps}{\Omega_{\eps}}
\newcommand{\Sigmaeps}{\Sigma_{\eps}}
\newcommand{\weakconv}{\rightharpoonup}
\newcommand{\ueps}{u_{\eps}}
\newcommand{\veps}{v_{\eps}}

\newcommand{\Feps}{F_{\eps}}

\newcommand{\lajeps}[1]{\lambda^{#1}_{\eps}}

\newcommand{\Ujeps}[1]{U^{#1}_{\eps}}

\begin{document}

\maketitle

\abstract{\noindent 
We consider a Neumann-Robin spectral problem in a perforated domain $\Omeps$. By homogenization techniques we find the suitable homogenized problem and we discuss the asymptotics of eigenpairs, as the size of the perforation tends to zero. Our results involve an approach based on Vi\v{s}\'ik lemma and the Mosco convergence of eigenspaces. We prove that eigenpairs of our problem converge to eigenpairs of the homogenized problem with rate $\sqrt{\eps}$.}\\

\begin{center}{\bf Keywords}\end{center}
{\noindent Homogenization, Perforated domains, $\Gamma$-convergence, Spectral problems.}

\section{Introduction}

In this paper we focus our attention on the homogenization of an elliptic spectral problem in a perforated domain. 
We consider a second order divergence form elliptic operator, defined on a periodically perforated domain $\Omeps \subset \Rn{d}$, where the small positive parameter $\e$ represents the microstructure period.
The structure of $\Omeps$ will be described throughly in section \ref{Problem statement}; in order to fix the ideas, we can consider the simplest case, defined as 
\[
\Omeps = \Omega \setminus \bigcup_{i \in I_\eps} B_\eps^i,
\]
where $\Omega$ is an open bounded set of $\Rn{d}$ with Lipschitz boundary and $B^i_\eps$, for $i \in I_\eps \subset \mathbb{Z}^d$ are the holes, obtained from a given  sufficiently smooth and closed set $B \subset Q=(0,1)^d$, by means of translations and homothety, as follows
\[
B^i_\eps= \eps (B+i), \quad I_{\eps} = \left\{i \in \mathbb{Z}^d:\ \eps(B+i) \subset \Omega \right\}.
\]
Hence, by construction, the boundary of perforated domain will be
\[
\partial \Omeps = \partial \Omega \cup \left( \bigcup_{i \in I_\eps} \partial B_\eps^i \right) = \partial \Omega \cup \Sigmaeps.
\]

In such a domain we consider elliptic PDEs and related spectral problems: in our case
\[
-\hbox{div}(a_\eps(x) \nabla u(x))=\lambda_\eps \ueps,
\]
where $a_\eps(x) = a(x/\eps)$ and $a \in \mathcal{M}^{d \times d}$ is a $Q$-periodic and symmetric matrix satisfying a standard ellipticity condition, so that the problem is actually equivalent to 
\begin{equation}\label{eqn}
-\triangle \ueps(x) = \lambda_\eps \ueps.
\end{equation} 
This type of problems has been treated since the 1970s: the reader can find many examples in books as \cite{A}, \cite{CD}, \cite{C-SJP}, \cite{SP}. In particular for our analysis, the crucial work  on spectral problems  by Vanninathan \cite{V} collects several results about the asymptotics of eigenpairs with Dirichlet, Neumann and Steklov boundary conditions. There, the author points out that the  behavior of the eigenpairs $(\lambda_\eps,\ueps)$ of problems in a perforated domain $\Omeps$, as $\eps \to 0$, strongly depends on the boundary conditions on $\partial \Omeps$.

Starting from this paper, many authors have worked on similar problems, changing boundary conditions or hypotheses on the geometry of the perforated domain, adding weight functions or analyzing localization effects. The boundary-value problem or spectral problem with  Fourier boundary condition was treated by several authors (\cite{V82}, \cite{CD88}, \cite{CSJP92}, \cite{CD97}, \cite{BCP98}, \cite{CP99} ). 

In our work we consider Fourier type boundary conditions with variable coefficients:

\begin{equation}\label{q}
\nabla \ueps(x) \cdot n_{\eps} = -q(x) \ueps(x),\, x \in \Sigmaeps, \quad \ueps(x)=0,\, x \in \partial \Omega.
\end{equation}
whose behavior depends on the assumptions on the weight function $q(x)$.
The problem was suggested by a work by Chiad\`o Piat, Pankratova, Piatniski \cite{CP-P-P}, where the authors consider  problem (\ref{eqn}), (\ref{q}) with $q \in C^2(\overline{\Omega})$ strictly positive and realizing its  global minimum at a unique point  $x_0 \in \Omega$. Moreover,  they assume the Hessian matrix in $x_0$ to be positive definite. Here,  on the contrary, we suppose that
\[
q(x)=\begin{cases}\displaystyle
      		0	& x \in K, \\
      		1	& x \in \Omega \setminus K,
	\end{cases}
\]
where $K \Subset \Omega$ is a compact set with non empty interior part $A$ and Lipschitz boundary; so that boundary conditions over the perforation $\Sigmaeps$ are of Neumann type, when weight function $q=0$, or Robin type, when $q=1$. 

A physical interpretation of the weight $q$ in our work is as the insulating power of the holes located inside $K$. Namely, the homogeneous Neumann boundary condition at  the boundary of the  holes $\Sigmaeps$ means that they represent completely insulating inclusions; the presence of the weight $q$, which is zero in $K$ and positive outside it, allows these inclusions to conduct only in region $\Omega \setminus K$ and to be insulating in $K$. Homogenization can describe the behavior of such a material when the number of these holes tends to infinity and their size tends to zero.

We consider the  functional $F_{\eps}: L^2(\Omega) \to [0,+\infty]$ defined as

\[
F_{\eps}(u)=	\begin{cases}
			\displaystyle \int_{\Omeps}|\nabla u|^2 + \int_{\Sigmaeps \setminus K} |u|^2	& u \in H^1_0(\Omeps; \partial \Omega), \\
			+\infty											& \text{otherwise},
			\end{cases}
\]
where 
\begin{equation}\label{def Heps}
H_{\eps}=H^1_0(\Omeps; \partial \Omega)=\left\{ u \in H^1(\Omeps): u=0\, \text{in}\, \partial \Omega \right\}
\end{equation}
is a Hilbert space, equipped with the scalar product
\[
(u,v)_{H_\eps}= \int_{\Omeps} \nabla u \nabla v dx.
\]

It is convenient to represent   the first eigenvalue through its variational characterization
\[
\lambda^1_{\eps} = \min \left\{ \Feps(u): u \in H_\eps, \int_{\Omeps} u^2 dx = 1\right\}.
\]
More generally,  we can describe any eigenvalue $\lambda_{\eps}^j$ in a variational way. by introducing a basis of eigenfunctions  $\ueps^i$  of our problem (\ref{eqn}), (\ref{q}),
and by  taking the minimum over the spaces 
\[
H^j_{\eps}=\left\{u \in H_{\eps}:  (u,\ueps^i)_{H_\eps}=0,\ i=1,\dots,j-1\right\}.
\]
in section \ref{The upperbound} we first prove equiboundedness of the first eigenvalue $\lambda^1_{\eps}$. Then we compute the $\Gamma$-limit of $\Feps$:
\[
\gammalim{F_{\eps}(u)}=F(u)=	\begin{cases}\displaystyle	
							\int_{\Omega} f^{\text{hom}}(\nabla u) dx,			& u \in H^1_0(A) \\
							+\infty									& \text{otherwise}.
							\end{cases}
\]
Here $f^{\text{hom}} : \Rn{d} \to [0,+\infty]$ is defined by
\[
f^{\text{hom}}(\xi) = \inf \left\{ \int_{Y} |\xi + \nabla u|^2 dx: \quad u \in H^1_{\text{per}}(\Rn{d}) \right\}.
\]
and gives  the homogenized spectral problem, with Dirichlet conditions,
\begin{equation}\label{intro pbhom}
\begin{cases}\displaystyle
-\text{div}(a^{\text{hom}} \nabla u)= |Y| \lambda u	&	u \in A\\
u=0										&	u \in \partial A,
\end{cases}
\end{equation}
where $a^{\text{hom}} \xi \xi = f^{\text{hom}}(\xi)$, $A$ is the interior part of $K$ and $Y=Q\setminus B$ is the perforated periodicity cell.

This $\Gamma$-convergence result, together with the equicoerciveness of $\Feps$ and the variational formulation of eigenvalues, implies the convergence
\[
\lambda_\eps^j \xrightarrow[\eps \to 0]{} \lambda^j,
\]
with $\lambda^j$ eigenvalues of the homogenized problem (\ref{intro pbhom}). In section \ref{convergence of eigenvalues} we also study the asymptotics of the eigenfunctions $\ueps^j$ with respect to  the homogenized ones $u^j$, proving the convergence of eigenspaces, in the sense of Mosco; moreover, we investigate  the rate of convergence of eigenfunctions in the $L^2$-norm, and we show that is of the order $\sqrt{\eps}$. This last result is obtained following a classical procedure similar to the one in \cite{CP-N-P}, that exploits Vi\v{s}\'ik lemma \cite{O-Y-S}.

\section{Problem statement}\label{Problem statement}

We consider the sets $Q=[0,1)^d$ and $E \subset \Rn{d}$, which is $Q$-periodic, open and connected, with Lipschitz boundary $\Sigma = \partial E$. We define the complement of $E$, $B=\Rn{d} \setminus E$, that represents the holes, assuming that $Q \cap E$ is connected and $Q\cap B \Subset Q$, so that $B$ consists of disjoint components. So we denote by $Y=Q \cap E$ the periodicity perforated cell and $\Sigma^0=Q \cap \partial B = Q \cap \Sigma$, the boundary of the hole.

\begin{figure}[!h]\centering
\includegraphics[width=0.4\columnwidth]{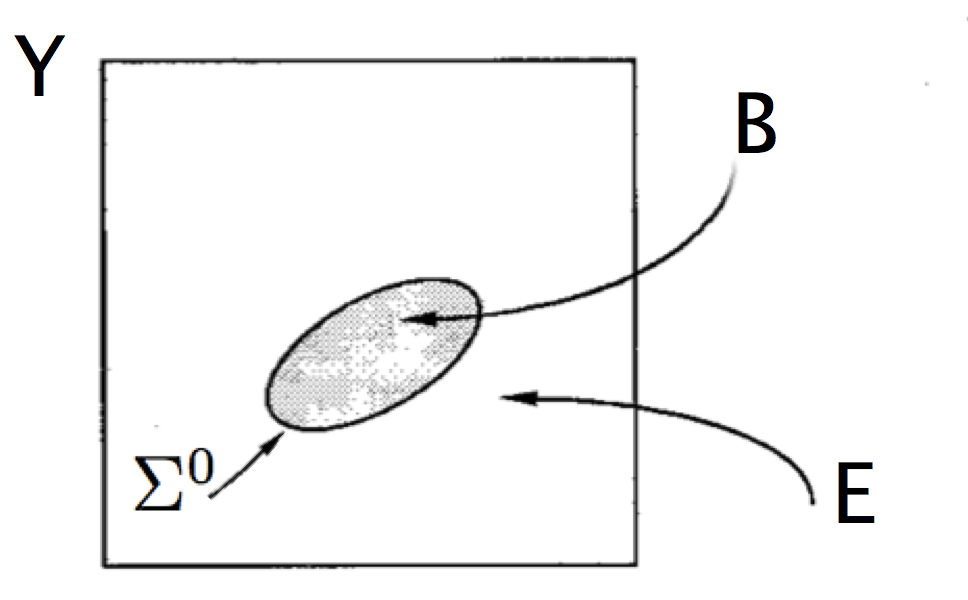}
\caption{A perforated cell.}\label{fig:perf cell}
\end{figure}

For every $i \in \mathbb{Z}^d$ and for fixed $\eps >0$, we denote $Y^i_{\eps}=\eps(i+Y)$, $\Sigma_{\eps}^i=\eps \Sigma \cap Y_{\eps}^i$, and $B_{\eps}^i=\eps B \cap Y_{\eps}^i$, that are respectively the periodicity cells, the boundary of the holes, and the holes themselves. Given a bounded open set $\Omega\subset \Rn{d}$, with Lipschitz boundary $\partial \Omega$, our perforated domain is

\begin{equation}\label{def omeps}
\Omega_{\eps}=\Omega \setminus \bigcup_{i \in I_{\eps}} B^i_{\eps},\ \ I_{\eps}=\left\{ i \in \mathbb{Z}^d: \  Y_{\eps}^i \subset \Omega \right\}.
\end{equation}

By the hypothesis on the connectedness of $Q\cap E$, we can assume that $\Omega_{\eps}$ is still connected; by definition of $I_{\eps}$ we get that the holes don't intersect $\partial \Omega$; we have
\[
\partial \Omega_{\eps}= \partial \Omega \cup \Sigma_{\eps},\ \ \ \Sigma_{\eps}= \bigcup_{i \in I_{\eps}} \Sigma_{\eps}^i.
\]

\begin{figure}[!h]\centering
\includegraphics[width=0.4\columnwidth]{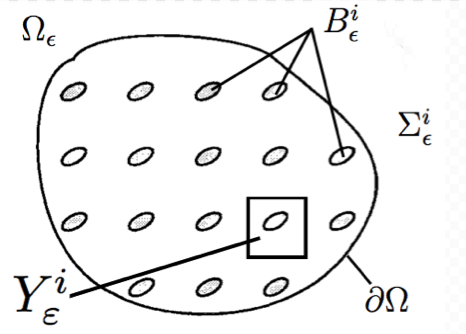}
\caption{The perforated domain $\Omeps$.}\label{fig:perf domain}
\end{figure}

A more general case, avoiding these last assumptions, can be treated as in \cite{A-CP-DM-P}.
In the perforated domain $\Omeps$ we consider the following spectral problem

\begin{equation}\label{Pbepsa}
\begin{cases}\displaystyle
-\text{div}(a_{\eps}(x) \nabla u_{\eps}(x)) = \lambda_{\eps} u_{\eps}(x),	& x \in \Omega_{\eps}, \\
a_{\eps}(x) \nabla u_{\eps}(x) \cdot n_{\eps} = - q(x) u_{\eps} (x), 			& x \in \Sigma_{\eps}, \\
u_{\eps}(x)=0,											& x \in \partial \Omega.
\end{cases}
\end{equation}

Here $a_{\eps}(x)=a(x/\eps)$, where $a(y)$ is a $d \times d$ matrix; $n_{\eps}$ is the outward unit normal at the boundary $\Sigma_{\eps}$, and $\cdot$ denotes the usual scalar product in $\Rn{d}$.

We will state the following hypothesis:

\begin{itemize}
\item[(H1)] $a(y)$ is a real symmetric matrix satisfying the uniform ellipticity condition
\[
\sum_{i,j = 1}^d a_{ij}(y) \xi_i \xi_j \geqslant \alpha |\xi|^2, \ \ \xi \in \Rn{d},
\]
for some $\alpha>0$.
\item[(H2)] The coefficients $a_{ij}(y)$ are in $L^{\infty}(\Rn{d})$ and $Q$-periodic.
\item[(H3)] The function $q(x)$ is defined as
\[
q(x) =	\begin{cases}\displaystyle
		0,	&	x \in K, \\
		1,	&	x \in \Omega \setminus K,
		\end{cases}
\]
where $K \Subset \Omega$ is a compact  subset of $\Omega$, with non empty interior $A=\dot{K}$ and Lipschitz boundary.
\end{itemize}

We can consider the weak formulation of problem (\ref{Pbepsa}), that is to find $\lambda_{\eps} \in \mathbb{C}$ (eigenvalues) and $u_{\eps} \neq 0$ (eigenfunctions) in the Hilbert space $H_\e$ defined in (\ref{def Heps}), such that

\begin{equation}\label{Pbepsaw}
\int_{\Omeps} a_{\eps}(x) \nabla u_{\eps} \cdot \nabla v\, dx + \int_{\Sigmaeps} q(x) u_{\eps} v\, d\sigma = \lambda_{\eps} \int_{\Omeps} u_{\eps} v \,dx, \ \ \ v \in H^1_0(\Omega).
\end{equation}

Let us start with a simpler case, where we consider the matrix $a_{i,j}(y) = \delta_{ij}$, so that we deal with the spectrum of the the Laplacian operator. We may easily generalize the results to our problem (\ref{Pbepsa}), using the ellipticity condition and boundedness of $a(y)$. So we can consider eigenpairs $(\lambda^j_{\eps}, u^j_{\eps})$, with $j \in \mathbb{N}$, of  problem

\begin{equation}\label{Pbeps}
\begin{cases}\displaystyle
-\triangle u^{\eps}(x) = \lambda^{\eps} u^{\eps}(x),	& x \in \Omega_{\eps}, \\
\nabla u^{\eps}(x) \cdot n_{\eps} = - q(x) u^{\eps} (x), 			& x \in \Sigma_{\eps}, \\
u^{\eps}(x)=0,											& x \in \partial \Omega,
\end{cases}
\end{equation}
that is, in the weak formulation, find $\lambda_{\eps} \in \mathbb{C}$ and $u_{\eps} \in H_{\eps}$, $u_{\eps} \neq 0$, such that

\begin{equation}\label{Pbepsw}
\int_{\Omeps} \nabla u_{\eps} \cdot \nabla v\, dx + \int_{\Sigmaeps} q(x) u_{\eps} v\, d\sigma = \lambda_{\eps} \int_{\Omeps} u_{\eps} v \,dx, \ \ \ v \in H_{\eps}.
\end{equation}

Let us define a linear operator $K_{\eps}$, whose spectrum will be related to the eigenvalues of problem (\ref{Pbeps}). So let consider the standard embedding operator
\[
J_{\eps}:H_{\eps} \to L^2(\Omeps),
\]
which is compact, due to the regularity of $\partial \Omeps$.

Now we take the operator
\begin{eqnarray}\label{def Ktildeps}
\tilde{K}_{\eps}: L^2(\Omeps)&	 \to&		 H_{\eps} \\ \notag
f&						\mapsto&	\tilde{K}_{\eps} f,
\end{eqnarray}
where $\tilde{K}_{\eps} f$ is the unique solution of the following problem

\begin{equation}\label{eq PbKeps}
\begin{cases}\displaystyle
-\triangle u^{\eps}(x) = f,							& x \in \Omega_{\eps}, \\
\nabla u^{\eps}(x) \cdot n_{\eps} = - q(x) u^{\eps} (x), 		& x \in \Sigma_{\eps}, \\
u^{\eps}(x)=0,									& x \in \partial \Omega,
\end{cases}
\end{equation}
that is, in weak formulation, a function $\ueps \in H_{\eps}$ satisfying

\begin{equation}\label{eq PbKeps weak}
\int_{\Omeps} \nabla \ueps \nabla v + \int_{\Sigmaeps \setminus K} \ueps v = \int_{\Omeps} f v,
\end{equation}
for any $v \in H_{\eps}$.

Note that, by Lax-Milgram theorem, at $\eps>0$ fixed, for any $f \in L^2(\Omeps)$, there exists a unique $\ueps \in H_{\eps}$ solving problem (\ref{eq PbKeps}) or, equivalently, (\ref{eq PbKeps weak}), so that $\tilde{K}_{\eps}$ is well defined. We will consider the operator
\begin{equation}\label{def Keps}
K_{\eps}: H_{\eps} \to H_{\eps}, \quad K_{\eps} = \tilde{K}_{\eps} \cdot J_{\eps}.
\end{equation}

\begin{lemma}\label{lemma spettro}
The operator $K_{\eps}: H_{\eps} \to H_{\eps}$ is positive, linear, compact and self-adjoint.
\end{lemma}

\begin{proof}
The proof of the linearity and continuity of $\tilde{K}_{\eps}$ follows from Lax-Milgram, the fact that $\tilde{K}_{\eps}$ is self-adjoint and positive is classical, see for example \cite{H}. Being $J_{\eps}$ the compact embedding operator, we simply get the thesis by composition.
\end{proof}

For spectral problem (\ref{Pbepsw}) we have the following classical result:

\begin{thm}\label{thm spettro}
For any $\eps>0$, the spectrum of problem (\ref{Pbepsw}) is real and consists of a countable set of values
\[
0<\lambda^1_{\eps}\leqslant \lambda^2_{\eps} \leqslant \cdots \leqslant \lambda^j_{\eps} \leqslant \cdots + \infty.
\]
Every eigenvalue has a finite multiplicity. The corresponding eigenfunctions normalized by
\[
\int_{\Omeps} u^i_{\eps} u^j_{\eps} \, dx = \delta_{ij},
\]
form a orthonormal basis in $L^2(\Omeps)$. Furthermore $\lambda^1_{\eps}$ is simple.
\end{thm}

\begin{proof}

By lemma $\ref{lemma spettro}$ and the general spectral theory, we have that the spectrum of the operator $K_{\eps}$ is made by a sequence of positive eigenvalues converging to zero:
\[
+\infty > \mu_{\eps}^1 \geqslant \mu_{\eps}^2 \geqslant \dots \geqslant \mu_{\eps}^j \geqslant \dots >0.
\]

Now observe that, if $\mu_{\eps}^j$ is an eigenvalue for $K_{\eps}$, i.e. there exists $\ueps^j \in H_{\eps}$ such that $K_{\eps} \ueps^j = \mu_{\eps}^j \ueps^j$, then
\begin{equation}
\begin{cases}\displaystyle
-\triangle \ueps^j(x) = \frac{1}{\mu_{\eps}^j} \ueps^j(x),	& x \in \Omega_{\eps}, \\
\nabla u_{\eps}^j(x) \cdot n_{\eps} = - q(x) u_{\eps}^j (x), 			& x \in \Sigma_{\eps}, \\
u_{\eps}(x)^j=0,											& x \in \partial \Omega,
\end{cases}
\end{equation}
so that $\lajeps{j}=\frac{1}{\mu_{\eps}^j}$ is an eigenvalue of problem (\ref{Pbeps}), hence we have that
\[
0<\lambda_1^{\eps}\leqslant \lambda_2^{\eps} \leqslant \cdots \leqslant \lambda_j^{\eps} \leqslant \cdots + \infty.
\]

Now we have to prove that $\lambda^1_\e$ is simple. First we show that if $\ueps^1$ is an eigenfunction associated to $\lajeps{1}$, then $\ueps^1$ doesn't change sign. To do this assume the contrary. Then $\ueps^{+}= \max \left\{ \ueps^1,0\right\}$ and $\ueps^{-}= \min \left\{ \ueps^1,0\right\}$ are non-trivial functions. Furthermore $\ueps^+$ and $\ueps^-$ are in $H_{\eps}$, so, by the variational characterization of the first eigenvalue,
\begin{equation}\label{eq u+-}
\lajeps{1} \leqslant \frac{\int_{\Omeps} |\nabla \ueps^+|^2+\int_{\Sigmaeps} q |\ueps^+|^2}{\int_{\Omeps}|\ueps^+|^2}, \quad \lajeps{1} \leqslant \frac{\int_{\Omeps} |\nabla \ueps^-|^2+\int_{\Sigmaeps} q |\ueps^-|^2}{\int_{\Omeps}|\ueps^-|^2}.
\end{equation}
Summing up these inequalities one has
\[
\lajeps{1}\left(\int_{\Omeps}|\ueps^+|^2 +\int_{\Omeps}|\ueps^-|^2 \right) \leqslant \int_{\Omeps} |\nabla \ueps^+|^2 +|\nabla \ueps^-|^2 + \int_{\Sigmaeps} q \left( |\ueps^+|^2 + |\ueps^-|^2 \right),
\]
and, by the fact that $\ueps^+ \ueps^- = 0$, we get
\[
\lajeps{1}\left(\int_{\Omeps}|\ueps^1|^2 \right) \leqslant \int_{\Omeps} |\nabla \ueps^1|^2+\int_{\Sigmaeps} q |\ueps^1|^2.
\]
But $\ueps^1$ is an eigenfunction associated to $\lajeps{1}$, hence this last inequality is actually an equality, and so are equations (\ref{eq u+-}). Then we have that $\ueps^+$ is a non negative solution of the equation $-\triangle \ueps = \lambda_{\eps} \ueps$ , with Neumann conditions on $\Sigmaeps$ and Dirichlet on $\partial \Omega$, that is zero at $\partial \Omega$ and it vanishes in the interior of $\Omega$ too: this contradicts the maximum principle, see \cite{H}, proposition IX.30.

Now assume that there exist two different and linearly independent eigenfunctions $\ueps^1$ and $\veps^1$ associated to $\lajeps{1}$; then, taking $c = \left( \int_{\Omeps}\veps^1\right)^{-1} \left(\int_{\Omeps} \ueps^1 \right)$, we have that $\ueps^1-c\veps^1$ is an eigenfunction too, with
\[
\int_{\Omeps} \ueps^1 - c\veps^1 = 0;
\]
therefore $\ueps^1-c\veps^1$ changes sign, and this contradicts our previous argument.

\end{proof}

Under hypothesis (H1), (H2), (H3), we want to study the asymptotic behavior of eigenpairs $(\lambda_{\eps}, u_{\eps})$, as $\eps \to 0$.

\section{Estimates for the first eigenvalue and determination of the limit problem}\label{The upperbound}

In this section we will show an upperbound for the first eigenvalue, at $\eps>0$ fixed, and we will consider the limit of $\lambda^1_\e$ as $\e \to 0$.

\begin{lemma}\label{upperbound}
For the first eigenvalue of problem (\ref{Pbepsw}), as $\eps \to 0$, we have the following inequality
\begin{equation}\label{upbound l1}
\limsup_{\eps \to 0} \lambda^1_{\eps}  \leqslant \alpha^1,
\end{equation}
where $\nu^{1}$ is the first eigenvalue of the Laplace operator defined over the set $A=\dot{K}$, with homogenous Dirichlet condition on $\partial A$:

\begin{equation}\label{pb lambda1}
\begin{cases}\displaystyle
-\triangle u(x)= \lambda u(x),	&	x \in A,\\
u(x)=0					&	x \in \partial A,
\end{cases}
\end{equation}
or, in the variational form,
\begin{equation} \label{lambda1}
\alpha^{1}= \inf_{u \in H^1_0(A)} \frac{\int_{A} |\nabla u|^2}{ \int_{A} |u|^2}.
\end{equation}

\end{lemma}

\begin{proof}
By Rayleigh equation we have

\begin{equation}\label{lambda1eps}
\lambda^1_{\eps} = \inf_{u \in H_{\eps}} \frac{\int_{\Omeps} |\nabla u(x)|^2 dx + \int_{\Sigmaeps} q(x) |u(x)|^2 d\sigma}{\int_{\Omeps} |u(x)|^2}.
\end{equation}

Let $u$ be a normalized solution of the minimum problem (\ref{lambda1}) on the set $A$:

\[
\frac{\int_{A} |\nabla u|^2}{\int_{A}|u|^2} = \alpha^1, \quad \int_{A} |u|^2=1, \quad u \in H^1_0(A).
\]

We can extend $u$ in $H^1_0(\Omega)$ by setting zero in $\Omega \setminus A$. Now define the function 

\[
u_{\eps}=\frac{u}{\| u \|_{L^2(\Omeps)}^2}, \quad u_{\eps} \in H_{\eps}.
\]

Since, as $\eps \to 0$, one has $\chi_{\Omeps} \rightharpoonup^* |Y| \chi_{\Omega}$, in $L^{\infty}$- weak*, where $|Y|$ is the $d$-dimensional measure of the perforated cell $Y$, we have that

\[
\| u \|^2_{L^2(\Omeps)}=\int_{\Omeps} |u|^2 = \int_{\Omega} |u|^2 \chi_{\Omeps} \xrightarrow[\eps \to 0]{} |Y|\int_{\Omega} |u|^2 =|Y| \int_{A} |u|^2 = |Y|.
\] 

We can use $u_{\eps}$ as a test function in the functional (\ref{lambda1eps}), remembering that $u_{\eps}$ is equal $0$ out of the set $A$, while $q=0$ inside $K$, getting
\[
\lambda^1_{\eps} \leqslant \frac{\int_{\Omeps} |\nabla u_{\eps}(x)|^2 dx + \int_{\Sigmaeps} q(x) |u_{\eps}(x)|^2 dx}{\int_{\Omeps} |u_{\eps}(x)|^2} 
=\frac{\frac{1}{\| u \|^2_{L^2(\Omeps)}}\int_{\Omeps} |\nabla u(x)|^2 dx}{\frac{1}{\| u \|^2_{L^2(\Omeps)}}\| u \|^2_{L^2(\Omeps)}} .
\]

Since
\[
\frac{\int_{\Omega} |\nabla u(x)|^2 \chi_{\Omeps} dx}{\| u \|^2_{L^2(\Omeps)}} \xrightarrow[\eps \to 0]{} \frac{|Y|\int_{\Omega} |\nabla u(x)|^2 dx}{|Y|}= \alpha^1,
\]
then we get the thesis.
\end{proof}

Now we want to find a limit for $\lambda_\e^1$ as $\e \to 0$. The basic idea is to consider the first eigenvalue of problem (\ref{Pbepsw}) as the minimum of a functional, depending on $\eps$, that is equicoercive, in order to exploit the $\Gamma$-convergence property of minimizing sequences to describe the limit of this eigenvalue. Let us define $F_{\eps}: L^2(\Omega) \to [0,+\infty]$, with

\begin{equation} \label{def Feps spectral}
F_{\eps}(u)=	\begin{cases}\displaystyle
			\int_{\Omeps}|\nabla u|^2 + \int_{\Sigmaeps \setminus K} |u|^2	& u \in H^1_0(\Omeps; \partial \Omega), \\
			+\infty											& \text{otherwise.}
			\end{cases}
\end{equation}

We have that
\[
\begin{split}
\lambda^1_{\eps} &= \inf_{u \in H_{\eps}} \frac{\int_{\Omeps} |\nabla u(x)|^2 dx + \int_{\Sigmaeps} q(x) |u(x)|^2 dx}{\int_{\Omeps} |u(x)|^2} \\
&= \inf_{\substack{u \in H_{\eps}, \\ \int_{\Omeps}|u|^2=1}} \int_{\Omeps} |\nabla u(x)|^2 dx + \int_{\Sigmaeps \setminus K} |u(x)|^2 dx= \inf_{\substack{u \in L^2(\Omega), \\ \int_{\Omeps}|u|^2=1}} F_{\eps}(u).
\end{split}
\]

Now consider the set $X_{\eps}= \{u \in  L^2(\Omega): \  \| u\|^2_{L^2(\Omeps)}=1\}$, and the function

\[
I_{X_{\eps}}(u)=	\begin{cases}\displaystyle
			0		&	u \in X_{\eps} \cap  L^2(\Omega), \\
			+ \infty	&	u \in L^2(\Omega) \setminus X_{\eps}.
			\end{cases}			
\]

Hence

\[
\lambda^1_{\eps} = \inf_{\substack{u \in L^2(\Omega), \\ \int_{\Omeps}|u|^2=1}} F_{\eps}(u) = \inf_{u \in L^2(\Omega)} \left[ F_{\eps}(u) + I_{X_{\eps}}(u) \right].
\]

\begin{rmk}\label{rmk Xeps}\rm
Observe that a natural limit of the constraint $X_{\eps}$, as $\eps \to 0$, is the set $X=\{u \in L^2(\Omega): \  \int_{\Omega} |u|^2 = 1/|Y| \}$. Infact, if $X_{\eps} \ni \ueps \to u$ in $L^2(\Omega)$, then

\[
1=\int_{\Omeps} |\ueps|^2=\int_{\Omega} |\ueps|^2\chi_{\Omeps} \xrightarrow[\eps \to 0]{} |Y|\int_{\Omega}|u|^2
\]
so that we get the condition $\int_{\Omega}|u|^2=1/|Y|$.
\end{rmk}

Now we can prove the following preliminary result:

\begin{lemma}\label{gamma lemma}
If the functional $F_{\eps}$, defined in (\ref{def Feps spectral}), $\Gamma$-converges in the strong $L^2(\Omega)$ topology to a functional $F$, then we have
\[
\gammalim{(F_{\eps}+I_{X_{\eps}})}=F+I_X,
\]
in the same topology, with $I_X$ defined as $I_{X_{\eps}}$:
\[
I_X (u)=	\begin{cases}\displaystyle
		0		&	u \in X, \\
		+ \infty	&	u \in L^2(\Omega) \setminus X.
		\end{cases}		
\]
\end{lemma}

\begin{proof}
Let us define the functionals $\mathcal{F}_{\eps}=F_{\eps}+I_{X_{\eps}}$ and $\mathcal{F}=F+I_X$.
\begin{itemize}
\item[i)]
We have to show that for every sequence $u_{\eps}$ converging to $u$, in the strong topology of $L^2(\Omega)$, one has
\begin{equation}\label{liminf F vincolo}
\mathcal{F}(u) \leqslant \liminf_{\eps \to 0} \mathcal{F}_{\eps}(u_{\eps}).
\end{equation}

We can suppose, possibly passing to a subsequence, that exists the\\ $\lim \mathcal{F}_{\eps}(u_{\eps}) < + \infty$, i.e. $u_{\eps} \in X_{\eps}$, so that $\mathcal{F}_{\eps}(u_{\eps})=F_{\eps}(u_{\eps})$. By hypothesis we know that $F_{\eps}$ $\Gamma$-converges to $F$, then $F(u) \leqslant \lim F_{\eps}(u_{\eps})<+\infty.$

We have

\[
\int_{\Omega} |u_{\eps}|^2=\int_{\Omeps}|u_{\eps}|^2+\int_{\Omega \setminus \Omeps}|u_{\eps}|^2=1+\int_{\Omega} |u_{\eps}|^2 \chi_{\Omega \setminus \Omeps}.
\]

Now, by the weak* convergence $\chi_{\Omega \setminus \Omeps} \rightharpoonup^* 1-|Y|$ and the strong $L^2(\Omega)$ convergence $u_{\eps} \to u$, one has

\[
\int_{\Omega}|u|^2 \xleftarrow[\eps \to 0]{} \int_{\Omega} |u_{\eps}|^2 \xrightarrow[\eps \to 0]{} 1+(1-|Y|)\int_{\Omega}|u|^2, 
\]
therefore
\[
\frac{1}{|Y|}=\int_{\Omega} |u|^2 \quad \Rightarrow \quad u \in X,
\]
so that
\[
\mathcal{F}(u)=F(u)  \leqslant \lim_{\eps \to 0} F_{\eps}(u_{\eps})  = \lim_{\eps \to 0} \mathcal{F}_{\eps}(u_{\eps}).
\]

\item[ii)]
We will prove that, for every $u \in H_{\eps}$, with $\mathcal{F}(u)< +\infty$, i.e. $u \in X$, there exists a sequence $u_{\eps}$, converging to $u$ in $L^2(\Omega)$, such that
\[
\mathcal{F}(u) \geqslant \limsup_{\eps \to 0} \mathcal{F}_{\eps}(u_{\eps}).
\]

Being $u \in X$, that is $\int_{\Omega}|u|^2= 1/|Y|$, one has $\gammalim{F_{\eps}}=F$, hence there exists a sequence $v_{\eps}$, converging in $L^2(\Omega)$ to $u$, such that
\[
F(u) = \lim_{\eps \to 0} F_{\eps} (v_{\eps}).
\]

Now define the sequence $u_{\eps}=v_{\eps}/\| v_{\eps} \|_{L^2(\Omeps)}^2$.We have
\[
u_{\eps}=\frac{v_{\eps}}{\int_{\Omega}|v_{\eps}|^2 \chi_{\Omeps}} \xrightarrow{L^2(\Omega)} \frac{u}{|Y|\int_{\Omega} |u|^2}=u.
\]

Observe that, by construction, $\| u_{\eps}\|^2_{L^2(\Omeps)}=1$, that is $u_{\eps} \in X_{\eps}$, and, by $\Gamma$-convergence, one has 
\[
\lim_{\eps \to 0} \mathcal{F}_{\eps}(u_{\eps})=\lim_{\eps \to 0} \frac{1}{\| \veps \|^2} \Feps (\ueps) = \frac{1}{|Y|\int_{\Omega}|u|^2}F(u)=\mathcal{F}(u).
\] 
\end{itemize}
\end{proof}

Thanks to Lemma \ref{gamma lemma} we can consider the $\Gamma$-convergence of $F_{\eps}$ only, ignoring the oscillating constraint $I_{X_{\eps}}$. In order to do this, we will follow the procedure used in \cite{A-CP-DM-P}. In our case it will be simpler, because, by our hypothesis on the perforated domain, the holes don't intersect the boundary of $\Omega$: $\partial \Omega \cap \Sigmaeps= \emptyset$.

\begin{rmk}\label{rmk extension}\rm
It is well known, see for example \cite{A}, \cite{CD}, \cite{C-SJP}, \cite{SP}, that, under the present assumptions on $\Omeps$, for every $\eps>0$, there exists a linear and continuous extension operator $T_{\eps}: H^1(\Omeps) \to H^1(\Omega)$ such that, for any $u \in H^1(\Omeps)$

\begin{itemize}
\item[i)] $T_{\eps}u_{\eps}={u}_{\eps}$ in $\Omeps$,
\item[ii)] $\|T_{\eps} u_{\eps} \|_{H^1(\Omega)} \leqslant c \|u_{\eps}\|_{H^1(\Omeps)}$,
\end{itemize}
where the constant $c>0$ depends on $Y$, but is independent of $\eps$.
\end{rmk}

We will use the following  lemma.

\begin{lemma}\label{lemma volsup}
Let $\Omega_\e^K=(\Omega\setminus K)\setminus\{\cup_i B^i_\e :  Y^i_\e\subseteq (\Omega\setminus K)  \}$, where $K=\overline A$ and $A$ is, by our hypothesis, a non empty open set $A\Subset \Omega$ with Lipschitz boundary. Let $ \Sigma^K_\e=\{\cup\partial B^i_\e :  Y^i_\e\subseteq(\Omega\setminus K)\}  $, so that $\partial\Omega^K_\e=\partial\Omega\cup\partial K\cup \Sigma^K_\e$.\\
Hence there exists two constants $c=c(K)$, independent of $\eps$, and $\eps_0>0$ such that for any $\eps < \eps_0$ and $w \in H^1_0(\Omeps;\partial \Omega)$

\begin{equation} \label{volsup}
\left| \frac{C^*}{\eps} \int_{\Omeps^K} |w|^2 dx - \int_{\Sigmaeps^K} |w|^2 d\sigma \right| \leqslant c(K) \int_{\Omeps^K} |\nabla w|^2, 
\end{equation} 
where $C^*=\frac{|\partial Y|}{|Y|}$.
\end{lemma}

\begin{proof}
We proceed as in proof of lemma (4.1) in \cite{CP-P-P}. Let $\chi \in H^1_{\text{per}}(Y)$ the solution of
\[
\begin{cases}\displaystyle
- div_y \chi (y)= C^*	&	y \in Y\\
\chi (y) \cdot n = -1	&	y \in \Sigma_0\\
\chi \in H^1_{\text{per}}(Y).
\end{cases}
\]

Then we consider its periodic extension over the whole $\Omeps$ and the rescaled function $\eps \chi (x / \eps)$: one has
\[
-\eps\; div_x \chi(x/\eps)= C^*.
\]

Multiplying by $w^2$, for any $w\in H_{\eps}$, and integrating over $\Omeps^K$, we get

\[
-\eps \int_{\Omeps^K} div \chi (x/\eps) w^2(x) dx =C^* \int_{\Omeps^K} w^2(x);
\]
integrating by part we have
\[
\begin{split}
&\eps \int_{\Omeps^K} \chi(x/\eps) \nabla (w^2) dx -\eps \int_{\Sigmaeps^K} \chi(x/\eps)  \cdot n w^2 d\sigma \\
&= C^* \int_{\Omeps^K} w^2 dx + \eps \int_{\partial K} \chi(x/\eps) \cdot n w^2 d\sigma.
\end{split}
\]

Now, being $\chi \in L^{\infty}(\Omeps)$, and $\chi(x/\eps) \cdot n = -1$ in $\Sigmaeps$, one has
\[
\left| \frac{C^*}{\eps} \int_{\Omeps^K} w^2 dx - \int_{\Sigmaeps^K} w^2 d\sigma \right| \leqslant \|\chi \|_{L^\infty} \left( \int_{\Omeps^K} |\nabla(w^2)| dx + \int_{\partial K} |w^2| d\sigma  \right)
\]

Finally note that, by Cauchy-Schwarz and Poincar{\'e} inequalities,
\[
\int_{\Omeps^K} |\nabla (w^2)| dx \leqslant 2 \int_{\Omeps^K} |w \nabla w| dx \leqslant c'(K) \int_{\Omeps^K } |\nabla w|^2,
\]
and, by trace inequality,
\[
\int_{\partial K} w^2  \leqslant  c''(K) \int_{\Omeps^K } |\nabla w|^2.
\]

Therefore

\[
\left| \frac{C^*}{\eps} \int_{\Omeps^K} w^2 dx - \int_{\Sigmaeps^K} w^2 d\sigma \right| \leqslant c(K) \int_{\Omeps^K } |\nabla w|^2.
\]

\end{proof}

Before considering the $\Gamma$-limit of $F_{\eps}(u)$, we can prove the following useful compactness property:

\begin{lemma} \label{lemma compactness}
Let $\ueps \in L^2(\Omeps)$ be a sequence such that $\Feps(\ueps) \leqslant c$ $\forall \eps >0$, then 
\begin{itemize}
\item[a)] up to subsequence, $T_{\eps}\ueps \to u \in H^1_0(\Omega)$ strongly in $L^2(\Omega)$ and weakly in $H^1(\Omega)$;
\item[b)] $u=0$ in $\Omega \setminus K$;
\item[c)] If $\int_{\Omeps} |\ueps - u_0|^2 \to 0$, as $\eps \to 0$, with $u_0 \in L^2(\Omega)$, then $u=u_0$ $L^2(\Omega)$-almost everywhere and the convergence $a)$ holds for the whole sequence $T_{\eps}\ueps$.
\end{itemize}
\end{lemma}

\begin{proof}
By the equiboundedness of the functional $\Feps(\ueps) \leqslant  c$, it follows that $\ueps \in H_{\eps}$ and $\int_{\Omeps}|\nabla \ueps|^2 \leqslant c$. By the extension property of $\Omeps$ in remark \ref{rmk extension}, we have that $T_{\eps}\ueps \in H^1_0(\Omega)$ and
\[
\int_{\Omega} |\nabla T_{\eps} \ueps|^2 \leqslant c_1 \int_{\Omeps}|\nabla \ueps|^2 \leqslant c_2.
\]
Hence, up to subsequence, there exists a function $u \in H^1_0(\Omega)$, such that $a)$ holds.

To prove $b)$ remember that also $\int_{\Sigmaeps^K} |\ueps|^2 \leqslant c$, then, by (\ref{volsup}) in Lemma \ref{lemma volsup}, it follows that, for any $\eps>0$,
\[
\int_{\Omeps^K} |\ueps|^2 \leqslant \eps c.
\]
Now note that
\[
\eps c\geqslant \int_{\Omeps^K} |\ueps|^2 = \int_{\Omega \setminus K} |T_{\eps} \ueps|^2 \chi_{\Omeps^K} \xrightarrow[\eps \to 0]{} |Y| \int_{\Omega \setminus K} |u|^2,
\]
so that, taking the limit as $\eps \to 0$, we get
\[
\int_{\Omega \setminus K} |u|^2=0.
\]

Finally, from hypothesis $\int_{\Omeps} |\ueps - u_0|^2 \to 0$ in $c)$,
it follows that
\[
\lim_{\eps \to 0} \int_{\Omega} |u - u_0|^2 = 0,
\]
and $c)$ is proved.
\end{proof}

\begin{rmk}\label{rmk equicoercive}\rm
From lemma \ref{lemma compactness} and remark \ref{rmk Xeps} we deduce that the sequence $\left\{ F_{\eps}+I_{X_{\eps}} \right\}_{\eps}$ is equicoercive with respect to the strong topology of $L^2(\Omega)$. Moreover, again by lemma \ref{lemma compactness}, we can deduce that if $\Feps$ $\Gamma$-converges to $F$ in $L^2(\Omega)$, then $F(u)=+\infty$ whenever $u \neq 0 $ in $\Omega \setminus K$.
\end{rmk}

We will use the following technical lemma, whose proof is classical:
\begin{lemma}\label{Dphidelta}
Let $A \subseteq \Rn{d}$ be an open bounded set, with Lipschitz boundary, and set $A^{\delta} = \left\{ x \in A: dist(x,\partial A)>\delta \right\}$, then there exists a constant $c>0$ such that, for every $u \in H^1_0(A)$, we have
\[
\int_{A\setminus A^{\delta}} |u|^2 dx \leqslant C \delta^2 \int_{A\setminus A^{\delta}} |\nabla u|^2 dx
\]
\end{lemma}

Now we can finally state the $\Gamma$-convergence result.

\begin{thm}\label{gammalimit}
Let $\Feps$ be defined by (\ref{def Feps spectral}). Then, for any $u \in L^2(\Omega)$, one has
\[
\gammalim{F_{\eps}(u)}=F(u)=	\begin{cases}	\displaystyle
							\int_{\Omega} f^{\text{hom}}(\nabla u) dx,			& u \in H^1_0(A) \\
							+\infty									& \text{otherwise}.
							\end{cases}
\]
in the strong topology of $L^2(\Omega)$, with $f^{\text{hom}} : \Rn{d} \to [0,+\infty]$ defined by

\begin{equation}\label{eq fhom}
f^{\text{hom}}(\xi) = \inf \left\{ \int_{Y} |\xi + \nabla u|^2 dx: \quad u \in H^1_{\text{per}}(\Rn{d}) \right\}.
\end{equation}
\end{thm}

\begin{proof}\ \\
\begin{center}
{$\Gamma \hbox{-}\liminf$ inequality.}\\
\end{center} 
Let $u_{\eps}$ be a sequence in $L^2(\Omeps)$ strongly converging to a function $u$. We have to prove that $F(u) \leqslant \liminf F_{\eps}(u_\eps)$; so, without loss of generality, we can suppose that $\liminf F_{\eps}(u_\eps) < +\infty$. By lemma \ref{lemma compactness}, we have $u \in H^1_0(\Omega)$, with $u=0$ in $\Omega \setminus K$, and the weak convergence $T_{\eps}\ueps \weakconv u$ in $H^1(\Omega)$. Therefore we conclude that $F(u)< +\infty$.

Now, by proposition 3.6 in \cite{A-CP-DM-P}, we know that

\[
F(u) \leqslant \liminf_{\eps \to 0} \int_{\Omeps} |\nabla \ueps|^2 dx.
\]

Hence, we conclude that

\[
\begin{split}
\liminf_{\eps} F_{\eps}(u_{\eps}) &= \liminf_{\eps} \int_{\Omeps}|\nabla u_{\eps}|^2 + \int_{\Sigmaeps \setminus K} |u_{\eps}|^2\\ 
&\geqslant  \liminf_{\eps} \int_{\Omeps}|\nabla u_{\eps}|^2 \geqslant F(u)= \int_{\Omega} f^{\text{hom}}(\nabla u) dx.
\end{split}
\]
\ \\
\begin{center}
$\Gamma \hbox{-}\limsup$ inequality.\\
\end{center}
We have to show that for any $u \in H^1_0(A)$ (if $u \in L^2(\Omega) \setminus H^1_0(A)$ the result is trivial) there exists a sequence $\ueps \in H_\e$, with $\ueps \to u$ in $L^2(\Omega)$, such that
\[
F(u) \geqslant \limsup_{\eps \to 0} \Feps(\ueps).
\]

Let consider a function $u \in H^1_0(A)$ and the zero extension $\tilde{u}$ of $u$ out of $A$, defined as

\[
\tilde{u}(x)=	\begin{cases}\displaystyle
			u(x)		&	x \in A\\
			0		&	x \in \Omega \setminus A,
			\end{cases}
\]
so that $\tilde{u} \in H^1_0(\Omega)$. Using the result in proposition 3.6 in \cite{A-CP-DM-P}, we can find a sequence $\ueps \in H^1(A \cap \Omeps) \cap L^2(A)$, with $\ueps \to \tilde{u}$ in $L^2(\Omega)$, such that 

\begin{equation} \label{recovery}
\limsup_{\eps \to 0} \int_{\Omeps \cap A} |\nabla \ueps|^2 dx \leqslant  \int_{A} f^{\text{hom}}(\nabla u) dx = \int_{\Omega} f^{\text{hom}}(\nabla \tilde{u}) dx.
\end{equation}

To construct our recovery sequence we fix constant $\delta > 0$ and a set $A^{\delta} = \left\{x \in A: dist(x,\partial A)>\delta\right\}$. Then we consider a cut-off function $\varphi \in \mathcal{C}^{\infty}_0(A)$, with $0 \leqslant \varphi \leqslant 1$, $spt (\varphi) \subseteq A$, $\varphi=1$ in $A^{\delta}$, $|\nabla \varphi| \le c/\delta$, and we take a new sequence defined as

\[
\veps(x)=\varphi(x)\ueps(x)=	\begin{cases}\displaystyle
						\ueps(x)			&	x \in A^{\delta}\\
						\varphi(x)\ueps(x)	&	x \in A \setminus A^{\delta}\\
						0				&	x \in \Omega \setminus A,
						\end{cases}
\]
so that $\veps \in H_{\eps}$.
We will use the following algebraic inequality:
\begin{equation}\label{algebra}
|a+b|^2 \leqslant (1+\eta) |a|^2 + \left( 1+\frac{1}{\eta} \right) |b|^2,
\end{equation}
for all $a,b,\eta \in \mathbb{R}$, with $\eta > 0$.

For our sequence $\veps$, since $spt (\veps) \subseteq A=\dot{K}$, we have

\begin{eqnarray*}
\Feps(\veps)&=&\int_{(\Omeps \cap A)} |\nabla \veps|^2+ \int_{(\Sigmaeps \setminus K) \cap A} |\veps|^2 = \\
&=&\int_{\Omeps \cap A^{\delta}} |\nabla \ueps|^2 + \int_{\Omeps \cap (A \setminus A^{\delta})} |\nabla \varphi \ueps + \varphi \nabla \ueps|^2.
\end{eqnarray*}

Consider the second term and use inequality (\ref{algebra}) and the regularity of $\varphi$:
\[
\begin{split}
 &\int_{\Omeps \cap (A \setminus A^{\delta})} |\nabla \varphi \ueps + \varphi \nabla \ueps|^2 \\
 &\leqslant  (1+\eta) \int_{\Omeps \cap (A \setminus A^{\delta})} |\varphi \nabla \ueps|^2 + \left(1+\frac{1}{\eta} \right) \int_{\Omeps \cap (A \setminus A^{\delta})} |\nabla \varphi \ueps|^2 \\
 &\leqslant  (1+\eta)\int_{\Omeps \cap (A \setminus A^{\delta})} |\nabla \ueps|^2 + \left(1+\frac{1}{\eta} \right)\frac{2}{\delta^2} \int_{\Omeps \cap (A \setminus A^{\delta})} \left(|\ueps-\tilde{u}|^2 + |\tilde{u}|^2 \right).
\end{split}
\]

Note that, by (\ref{recovery}), $\int_{\Omeps \cap A} |\nabla \ueps|^2\leqslant c$, so that

\[
(1+\eta)\int_{\Omeps \cap (A \setminus A^{\delta})} |\varphi \nabla \ueps|^2 \leqslant \int_{\Omeps \cap (A \setminus A^{\delta})} |\nabla \ueps|^2 + \eta c.
\]

Now, by the convergence $\ueps \to \tilde{u}$, we have $1/\delta^2\int_{\Omeps \cap (A \setminus A^{\delta})}|\ueps-\tilde{u}|^2=o(1)$, as $\eps \to 0$, with $\delta$ fixed, and, by lemma \ref{Dphidelta}, being $u \in H^1_0(A)$, and $u=\tilde{u}$ in $\Omeps \cap (A \setminus A^{\delta})$, we get
\[
\frac{1}{\delta^2} \int_{\Omeps \cap (A \setminus A^{\delta})} |u|^2 \leqslant c' \frac{1}{\delta^2} \delta^2 \int_{\Omeps \cap (A \setminus A^{\delta})} |\nabla u|^2 \leqslant c' \int_{A \setminus A^{\delta}} |\nabla u|^2,
\]
that tends to $0$ as $\delta \to 0$. Hence we have, for any $\delta>0$, $\eta > 0$,

\[
\begin{split}
\limsup_{\eps \to 0} \Feps(\veps) &\leqslant \limsup_{\eps \to 0}  \left[ \int_{\Omeps \cap A^{\delta}} |\nabla \ueps|^2 +  \int_{\Omeps \cap (A \setminus A^{\delta})} |\nabla \ueps|^2 + \eta c \right.\\
&\left. +\left(1+ \frac{1}{\eta}\right) c' \int_{A \setminus A^{\delta}} |  \nabla u|^2\right] \\
&\leqslant \left[\limsup_{\eps \to 0} \int_{\Omeps \cap A} |\nabla \ueps|^2 \right] +\eta c + \left(1+ \frac{1}{\eta}\right) c' \int_{A \setminus A^{\delta}} |  \nabla u|^2.
\end{split}
\]
 
Taking the limit first as $\delta \to 0$ and then as $\eta \to 0$, we have

\[
\limsup_{\eps \to 0} \Feps(\veps) \leqslant \limsup_{\eps \to 0} \int_{\Omeps \cap A} |\nabla \ueps|^2,
\]
so that, using (\ref{recovery}), we finally get
\[
\limsup_{\eps \to 0} \Feps(\veps) \leqslant \int_{A} f^{\text{hom}} (\nabla u)^2 = F(u)
\]
\end{proof}
As a consequence of the gamma convergence result, we can consider the differential equation associated to the Euler equation defined by (\ref{eq fhom}): this means that our limit homogenized problem will be

\begin{equation}\label{Pb hom}
\begin{cases}\displaystyle
-\text{div}(a^{\text{hom}} \nabla u(x))= |Y| \lambda u	&	x \in A\\
u(x)=0									&	x \in \partial A,
\end{cases}
\end{equation}
where $a^{\text{hom}} \xi \xi = f^{\text{hom}}(\xi)$.

\begin{corollary}\label{cor gammalimit}
Let $\lambda_\e^1$ and $\lambda^1$ be the first eigenvalues of problem (\ref{Pbeps}) and (\ref{Pb hom}) respectively. Then
\begin{equation}\label{eq conv minimi}
\lim_{\eps \to 0} \lambda^1_{\eps} = \lambda^1
\end{equation}
\end{corollary}

\begin{proof}
First of all note that one has
\[
\lambda^1= \min_{u \in L^2(\Omega)} (F(u)+I_{X}(u)).
\] 
In theorem \ref{gammalimit} we proved that $\Feps \xrightarrow[]{\Gamma(L^2(\Omega))} F$. From lemma \ref{gamma lemma}, it follows that
\[
\Feps + I_{X_{\eps}} \xrightarrow[]{\Gamma(L^2(\Omega))} F + I_{X}.
\]

By remark \ref{rmk equicoercive} we know that $(\Feps + I_{X_{\eps}})_{\eps}$ is equicoercive. Therefore we immediately obtain (\ref{eq conv minimi}).
\end{proof}

\begin{rmk}\label{rma pb cella}\rm
It will be useful in the sequel to underline the relationship between equation $(\ref{eq fhom})$ and the associated problem on the periodicity perforated cell $Y$: the solution $w_{\xi}$ of the minimum problem defined by $f^{\text{hom}}(\xi)$ is in fact of type $w_{\xi}=\xi \cdot \chi$, where $\chi$ is the vector whose components solve the equation
\begin{equation}\label{eq cella}
\begin{cases}\displaystyle
	\triangle \chi^i (x)= 0						&	x \in Y \\
	\frac{\partial \chi}{\partial \nu}(x) + n^i=0	&	x \in \Sigma^0 \\
	\chi \in H^1_{\text{per}}(Y)
\end{cases}
\end{equation}
where $n=n_1, \dots, n_d$ is the external normal vector to $Y$ over $\Sigma^0$.

\end{rmk}

\section{Convergence for eigenvalues and \\eigenfunctions of higher order.}\label{convergence of eigenvalues}

In this section we will consider the limit of $\lambda^j_{\eps}$ and $\ueps^j$, for any $j \in \mathbb{N}$, as $\eps$ goes to 0, proving that for any $\lambda$, eigenvalue of the limit problem (\ref{Pb hom}), there exists $\lambda^j_{\eps}$, eigenvalue of the problem on the perforated domain (\ref{Pbeps}), such that
\[
|\lambda-\lambda^j_{\eps}| \xrightarrow[\eps \to 0]{} 0,
\]
therefore, in theorem \ref{thm 2}, we will state an estimate for this convergence.

In this section we will assume that the perforation $B$ is a $\mathcal{C}^2$ set. Before studying the behavior of eigenvalues, as $\eps \to 0$, we present the following statement

\begin{lemma}\label{lower and upper bound}
For any $j \in \mathbb{N}$, there exist two positive constants $c_j$ and $c$, independent from $\eps$, and a constant and $\eps_0>0$, such that
\begin{equation}\label{eq bounds}
c \leqslant \lajeps{j} \leqslant c_j \quad \forall \eps< \eps_0
\end{equation}
\end{lemma}

\begin{proof}\ \\
\underline{\emph{Lower bound.}}

By theorem \ref{thm spettro}, it suffices to prove the inequality $c \leqslant \lajeps{j}$, for $\lajeps{1}$, being $\lajeps{1}\leqslant \lajeps{2} \leqslant \dots \leqslant \lajeps{j}$. By corollary \ref{cor gammalimit}, we have that $\lim_{\eps \to 0} \lajeps{1}=\lambda_1>0 $, then by the theorem of permanence of sign, there exists $\eps_0$ such that, for any $\eps < \eps_0$, we have $\lajeps{1}>0$, i.e. there exists $c>0$ such that $\lajeps{1} \geqslant c$.\\ \ \\
\underline{\emph{Upper bound.}}

Let us take $\varphi_i \in C^{\infty}_0(A)$, for $i=1,\dots,j$, a set of non-zero functions with disjoint supports. We extend by zero out of A, obtaining $\varphi_i \in C^{\infty}_0(\Omega)$. Since these functions are orthogonal in $H_{\eps}$, there is a non-trivial linear combination $\psi_{\eps}=\gamma_{\eps}^1\varphi_1 + \dots + \gamma_{\eps}^j \varphi_j$ such that
\[
(\psi_{\eps},\ueps^1)_{H_{\eps}}=\dots=(\psi_{\eps},\ueps^{j-1})_{H_{\eps}}=0.
\]

Then $\psi_{\eps}$ is a competitor for the minimum problem defined by $\lajeps{j}$, so that
\[
\begin{split}
\lajeps{j} &\leqslant \frac{\int_{\Omeps} |\nabla \psi_{\eps}|^2+\int_{\Sigmaeps \setminus K} |\psi_{\eps}|^2}{\int_{\Omeps}|\psi_{\eps}|^2}\\
&= \frac{\sum_{i=1}^{j} (\gamma_{\eps}^i)^2 \left( \int_{\Omeps} |\nabla \varphi_i|^2 + \int_{\Sigmaeps \setminus K} |\varphi_i|^2 \right)}{\sum_{i=1}^{j} (\gamma_{\eps}^i)^2 \int_{\Omeps} \varphi_i^2}.
\end{split}
\]
Now, being
\[
c'_j= \sup_{1\leqslant i \leqslant j} \int_{\Omeps} |\nabla \varphi_i|^2, \quad c''_j =  \sup_{1\leqslant i \leqslant j} \int_{\Omeps} | \varphi_i|^2,
\]
one has
\[
\lajeps{j}\leqslant \frac{c'_j}{c''_j}=c_j.
\]
\end{proof}

Before stating the first result, concerning the convergence of eigenvalues and eigenspaces, we want to present, for the reader convenience, the definition of a particular type of convergence, that we will use in theorem \ref{thm 1}.

\begin{definition}\label{def Mosco}
Let $\{ S_j\}_{j}$ be a sequence of convex subsets of a reflexive Banach space X. We say that $\{ S_j\}_{j}$ Mosco-converges to the set $S$, writing
\[
S_j \xrightarrow{M} S,
\]
if the following relation is satisfied:
\begin{equation}\label{eq. Mosco}
w-\limsup_{j \to +\infty} S_j = S = s-\liminf_{j \to +\infty} S_j.
\end{equation}

By $w-\limsup_{j} S_j$ we denote the set of $x \in X$ for which there exists a sequence  $x_j \weakconv x$ weakly and such that $x_j \in S_j$ frequently, i.e. for infinitely many indices $j \in \mathbb{N}$. By $s-\liminf_j S_j$ we mean the set of $x \in X$ for which there exists a sequence $x_j \to x$ strongly and such that $x_j \in S_j$ definitively.
\end{definition}

\begin{rmk}\label{rmk mosco def}\rm
To show (\ref{eq. Mosco}) it suffices to prove
\begin{equation}\label{diseq. Mosco}
w-\limsup_{j \to +\infty} S_j \subseteq S \subseteq s-\liminf_{j \to +\infty} S_j,
\end{equation}
in fact the following relation is always satisfied:
\[
s-\liminf_{j \to +\infty} S_j \subseteq w-\limsup_{j \to +\infty} S_j.
\]
\end{rmk}

We will use the Urysohn property for convex sets, that we recall here without proof, see for example \cite{M}:

\begin{propr}\label{prop ury}
Let $\left\{S_j\right\}_j$, $S$ be a sequence of convex subsets of a reflexive Banach space X. Then $S_j \xrightarrow{M} S$ if and only if for every subsequence $S_{j_k}$ there exists a further subsequence $S_{j_{k_l}}$ that Mosco converges to S. 
\end{propr}

Before showing our first result on the convergence of eigenvalues and eigenspaces, we state a classical property of Gamma convergence:
%
%
%

\begin{lemma}\label{cor Geps}
Consider the functional $\Feps$ described in (\ref{def Feps spectral}), and its $\Gamma$-limit $F$; define the new functional
\begin{equation}\label{eq def Geps}
G_{\eps}(v)=\Feps(v)-2\lambda_{\eps} \int_{\Omeps} \ueps v, 
\end{equation}
where $\lambda_{\eps}\to\lambda$ and $\ueps \to u$ strongly in $L^2(\Omega)$ as $\eps \to 0$. Then
\begin{equation}\label{eq gammaconv Geps}
G_{\eps}(v) \xrightarrow{\Gamma} G(v)= F(v) - 2\lambda |Y| \int_{\Omega} u v.
\end{equation}
\end{lemma}

\begin{proof}
First of all note that, by the weak-strong convergence, for any sequence $\left\{\veps\right\}_{\eps}$ in $L^2(\Omega)$, such that $\veps \to v$ strongly in $L^2(\Omega)$ as $\eps \to 0$, one has
\[
\lambda_{\eps} \int_{\Omeps} \veps u = \lambda_{\eps} \int_{\Omega} \veps u \chi_{\Omeps} \to \lambda |Y| \int_{\Omega} v u,
\]
so that 
\begin{equation}\label{eq Geps}
\Feps(u) + \lambda_{\eps} \int_{\Omeps} \veps u \xrightarrow[\eps \to 0]{} F(u) + |Y|\lambda \int_{\Omega} v u.
\end{equation}
\ \\
\begin{center}
{$\Gamma \hbox{-}\liminf$ inequality.}\\
\end{center} 
Let $v$ be in the domain of $G$, i.e. in $H^1_0(A)$, the domain of $F$, and consider its zero extension out of $A$ in $H^1_0(\Omega)$; Let $\veps \to v$ strongly in $L^2(\Omega)$, with $G_{\eps}(\veps)<C$. This means that $\veps$ is bounded in $H_{\eps}$ and, by lemma \ref{lemma compactness} we have
\[
T_{\eps}\veps \to v
\]
weakly in $H^1(\Omega)$, up to subsequence, and, by Rellich theorem, strongly in $L^2(\Omega)$. Since $\Feps \xrightarrow{\Gamma} F$, we have
\[
\liminf_{\eps} \Feps(\veps) \geqslant F(v),
\]
and, by equation (\ref{eq Geps}),
\[
\liminf_{\eps} G_{\eps} (\veps) = \liminf_{\eps} \Feps(\veps) - 2 \lambda_{\eps} \int_{\Omeps} \ueps \veps \geqslant
\]
\[
\ge F(v) - 2 |Y| \lambda \int_{\Omega} u v = G(v).
\]
\ \\
\begin{center}
{$\Gamma \hbox{-}\liminf$ inequality.}\\
\end{center} 
The Gamma limsup inequality follows directly from the $\Gamma$-convergence of $\Feps$ to $F$ and from equation (\ref{eq Geps}).

\end{proof}

We can now state the following result, whose proof follows a classic procedure, that can be find, for example, in \cite{A}:

\begin{thm}\label{thm 1}
Let $(\lajeps{j},\ueps^j)$ and $(\lambda^j,u^j)$ be the eigenpairs of problems (\ref{Pbeps}) and (\ref{Pb hom}), respectively. Then
\begin{itemize}
\item[1)] $\lajeps{j} \to \lambda^j$ as $\eps \to 0$, for every $j \in \mathbb{N}$;
\item[2)] if $\lambda^j$ has multiplicity $m_j$ and $\lambda^j=\lambda^{j+1}=\dots=\lambda^{j+m_j-1}$, and we set
\begin{equation}\label{eq span}
S^j_{\eps}= \text{span} \left[ T_{\eps}\ueps^j,\dots,T_{\eps}\ueps^{j+m_j-1} \right], \quad S^j=\text{span} \left[ u^j,\dots,u^{j+m_j-1} \right],
\end{equation}
then 
\[
S^j_{\eps} \xrightarrow[\eps \to 0]{M} S^j
\]
in $L^2(\Omega)$, for every $j \in \mathbb{N}$.
\end{itemize}
\end{thm}

Theorem \ref{thm 1} shows that any eigenvalue of the homogenized problem (\ref{Pb hom}) is the limit, as $\eps \to 0$, of the corresponding eigenvalue of the problem (\ref{Pbeps}), in the perforated domain, and the same is for any eigenspace, in the sense of Mosco. Our last result gives the rate of this convergence. In order to obtain it, we will use many technical tools, that we formulate in the sequel.

\begin{lemma}\label{lemma visik}
Let $H$ be a Hilbert separable space and $A:H \to H$ a linear compact self-adjoint operator. Suppose that there exist two real numbers $\mu$, $\alpha$ and a vector $u \in H$, with $\|u\|_H=1$ and
\[
\|Au - \mu u \|_H< \alpha.
\]
Then there is an eigenvalue $\mu_j$ of the operator $A$, such that
\begin{itemize}
  \item[i)] $|\mu_j-\mu|<\alpha$;
  \item[ii)] for any $d>\alpha$ there exists a vector $\tilde{u}$ in the eigenspace associated to eigenvalues $\mu_k \in [\mu_j-d,\mu_j+d]$, with $\| \tilde{u}\|_H=1$, such that
  \[
  \| u-\tilde{u} \|_H<\frac{2 \alpha}{d}.
  \]
\end{itemize}
\end{lemma}

This lemma is often known in the literature as Vi\v{s}\'ik lemma; for the proof see for example \cite{O-Y-S}.

We will use the following trace type inequality, whose proof follows from the classical Poincar{\'e}-Wirtinger and trace inequalities, see \cite{B,B11}:

\begin{lemma}\label{lemma trace}
For any $u \in H_{\eps}$ one has
\begin{equation}\label{eq trace}
\int_{\Sigmaeps} |u|^2 \leqslant c \left(\eps^{-1} \int_{\Omeps} |u|^2 + \eps \int_{\Omeps} |\nabla u|^2 \right).
\end{equation}
\end{lemma}

Using this lemma \ref{lemma trace} we can easily prove the following

\begin{propr}\label{prop norme equiv}
For any $u,v \in H_{\eps}$, define the norm

\begin{equation}\label{eq norma eps}
\| u \|^2_{\eps} = \int_{\Omeps} |\nabla u|^2 + \int_{\Sigmaeps \setminus K} |u|^2,
\end{equation}
coming from the scalar product

\begin{equation}\label{eq def aeps}
a_{\eps}(u,v)= \int_{\Omeps} \nabla u \nabla v + \int_{\Sigmaeps \setminus K} u v.
\end{equation}

Hence there exists a constant $c \in \mathbb{R}$ such that

\begin{equation}\label{eq norme equiv}
\|u\|_{H_{\eps}} \leqslant \|u \|_{\eps} \leqslant c \eps^{\frac{1}{2}} \|u\|_{H_{\eps}}.
\end{equation}

\end{propr}

We finally state the last preliminary tool; see for example \cite{C-P-S} for the proof,

\begin{lemma}\label{lemma periodic}
Let $\chi \in L^{\infty}_{\text{per}}(Y)$ be such that
\[
\int_Y \chi(y) dy = 0.
\]
There exists a constant $c>0$ such that, for any $u,v \in H^1_0(\Omega)$,

\begin{equation}\label{eq periodic}
\left|\int_{\Omega}\chi \left(\frac{x}{\eps}\right) u v dx \right| \leqslant c \eps \|u\|_{H^1_0(\Omega)} \|v\|_{H^1_0(\Omega)}.
\end{equation}

\end{lemma}

Now we can state the result on the rate of convergence of eigenvalues and correspondent eigenfunctions (considered with their multiplicity).

\begin{thm}\label{thm 2}
Let $\lambda^j$, $j \in \mathbb{N}$, be an eigenvalue of problem (\ref{Pb hom}) of multiplicity $m_j$: 
\[
\lambda^{j-1}<\lambda^j=\lambda^{j+1}=\dots=\lambda^{j+m_j-1}<\lambda^{j+m_j}.
\]
Let $(\lambda^j_{\eps},\ueps^j)_j$ be the eigenpairs of problem (\ref{Pbeps}) on the perforated domain. Then there exist orthogonal matrices $M_{\eps} \in \mathcal{M}^{m_j \times m_j}$ and constants $\eps_j$, $C_j$ such that, for any $\eps < \eps_j$,

\begin{equation}\label{eq conv autof approx}
\|U^{j+l-1}_\eps - \sum_{k=1}^{m_j} M_{\eps}^{lk} u_{\eps}^{j+k-1}\|_{H_{\eps}} \leqslant C_j \sqrt{\eps}, \quad l=1,\dots,m_j,
\end{equation}

\begin{equation}\label{eq conv autof}
\|u^{j+l-1} - \sum_{k=1}^{m_j} M_{\eps}^{lk} T_\eps u_{\eps}^{j+k-1}\|_{L^2(\Omega)} \leqslant C_j \sqrt{\eps}, \quad l=1,\dots,m_j,
\end{equation}
with 
\begin{equation}\label{def Ujeps}
U^{j}_{\eps}(x)= u^j(x)+\eps \chi \left(\frac{x}{\eps} \right) \nabla u^j(x),
\end{equation}
here $\chi$ is a solution of the cell problem (\ref{eq cella}).
\end{thm}

\begin{rmk}\rm
Observe that the function $\sum_{k=1}^{m_j} M_{\eps}^{lk} u_{\eps}^{j+k-1}$ in $(\ref{eq conv autof})$ belongs to the set $S^j_\eps$, defined in (\ref{eq span}); this means, being $u^{j+l-1} \in S^j$, that $S^j_{\eps}$ converges to $S^j$ in the sense of Mosco, moreover, the rate of this convergence is $\sqrt{\eps}$.
\end{rmk}

\begin{proof}
The proof, that follows from lemma \ref{lemma visik}, will be obtained through these three steps:
\begin{itemize}
  \item[Step 1.] We prove the following fundamental estimate, that involves the operator $K_{\eps}$, defined in (\ref{def Keps})
  \begin{equation}\label{estimate Keps}
  \|K_{\eps} \Ujeps{j} - \frac{1}{\lambda^j}\Ujeps{j} \|_{H_\eps} \leqslant c^j \sqrt{\eps},
  \end{equation}
  in the simpler hypothesis of $u^j \in C^{\infty}_{0}(A)$.
  \item[Step 2.] Applying lemma \ref{lemma visik}, we prove (\ref{eq conv autof approx}) and (\ref{eq conv autof}), and discuss the case of $\lambda^j$ of multiplicity $m_j$. 
  \item[Step 3.] We generalized the proof for $u^j \in H^1_0(\Omega)$.
\end{itemize}
\underline{\emph{Step 1.}}

Here we assume that $\partial A \in C^{2,\alpha}$, for $\alpha >0$, so that $u^j \in C^2(\overline{A})$ and, by hypothesis $u^j \in C^{\infty}_{0}(A)$, we get $\Ujeps{j} \in H^1_0(A)$. We have, using definition (\ref{eq def aeps}),
\[
\begin{split}
 \|K_{\eps} \Ujeps{j} - \frac{1}{\lambda^j}\Ujeps{j} \|_{\eps} &= \sup_{\substack{\varphi \in H_{\eps} \\ \|\varphi\|_{\eps}=1}} \left| a_{\eps} \left(K_{\eps} \Ujeps{j} - \frac{1}{\lambda^j}\Ujeps{j}, \varphi \right) \right|  \\
 &=\sup_{\substack{\varphi \in H_{\eps} \\ \|\varphi\|_{\eps}=1}} \left| a_{\eps}\left(K_{\eps} \Ujeps{j}, \varphi\right) - a_{\eps}\left( \frac{1}{\lambda^j}\Ujeps{j}, \varphi \right) \right|.
\end{split}
\]

Now, being $K_{\eps} \Ujeps{j}$ the solution of problem (\ref{eq PbKeps}), with $f=\Ujeps{j}$, we have
\begin{equation}\label{first term}
a_{\eps}\left(K_{\eps} \Ujeps{j}, \varphi\right) = \int_{\Omeps} \Ujeps{j} \varphi,
\end{equation}
and
\begin{equation}\label{second term}
a_{\eps}\left( \frac{1}{\lambda^j}\Ujeps{j}, \varphi \right) = \frac{1}{\lambda^j} \left(\int_{\Omeps} \nabla \Ujeps{j} \nabla \varphi + \int_{\Sigmaeps \setminus K} \Ujeps{j} \varphi \right)= \frac{1}{\lambda^j} \int_{\Omeps} \nabla \Ujeps{j} \nabla \varphi,
\end{equation}
because $u^j \in C^{\infty}_0(A)$.

Hence, for the first term (\ref{first term}), using Cauchy-Schwarz inequality, in the hypothesis of $\nabla u^j \in C^{\infty}_0(A)$, we have
\[
\begin{split}
 \int_{\Omeps} \Ujeps{j} \varphi&= \int_{\Omeps} \left( u^j+\eps \chi_{\eps} \nabla u^j\right) \varphi = \int_{\Omeps \cap A} u^j \varphi + \eps \int_{\Omeps} \chi_{\eps} \nabla u^j\varphi  \\
 &\leqslant \int_{\Omeps \cap A} u^j \varphi + \eps C \|\varphi \|_{L^{2}(\Omeps)} \\
 &= |Y| \int_{A} u^j T_{\eps}\varphi + \int_A \left( \chi_{\Omeps} - |Y|\right) u^j T_{\eps}\varphi + \eps C \|\varphi \|_{L^{2}(\Omeps)},
\end{split}
\]
and, using lemma \ref{lemma periodic}, the continuity of $T_\e$ and equation (\ref{eq norme equiv}), we get
\[
\begin{split}
\int_{\Omeps} \Ujeps{j} \varphi &\le |Y| \int_{A} u^j T_{\eps}\varphi + C \eps \|u^j \|_{H^1_0(A)} \|T_{\eps} \varphi\|_{H^1_0(\Omega)} + \eps C \|\varphi \|_{L^{2}(\Omeps)}\\
&\leqslant |Y| \int_{A} u^j T_{\eps}\varphi + c_1 \eps^{1/2}\| \varphi \|_{\eps}.
\end{split}
\]

%

On the other hand, for the second term (\ref{second term}), using same estimates of the first one, we get
\begin{equation*}
\begin{split}
&-\frac{1}{\lambda^j} \int_{\Omeps} \nabla \Ujeps{j} \nabla \varphi \\
&= -\frac{1}{\lambda^j} \int_{\Omeps \cap A} \left( \nabla u^j +\eps \nabla \chi\left(\frac{x}{\eps}\right)\frac{1}{\eps} \nabla u^j \right) \nabla \varphi - \frac{1}{\lambda^j}  \int_{\Omeps} \eps \chi_{\eps} D^2 u^j \nabla\varphi \\
&\leqslant -\frac{1}{\lambda^j} \int_{\Omeps \cap A} \left( \nabla u^j +\nabla_y \chi\left(y\right) \nabla u^j \right) \nabla \varphi + \frac{C}{\lambda^j}\eps\| \varphi \|_{H_{\eps}} \\ 
&\leqslant -\frac{1}{\lambda^j} \int_{\Omeps \cap A} \left( \nabla u^j +\nabla_y \chi\left(y\right) \nabla u^j \right) \nabla \varphi + \frac{c_2}{\lambda^j} \eps^{1/2}\|\varphi \|_{\eps},
\end{split}
\end{equation*}
%

%
%
%
where $y=x/\e$. Note that
\[
\begin{split}
&- \frac{1}{\lambda^j} \int_{\Omeps \cap A} \left( \nabla u^j +\nabla_y \chi\left(y\right) \nabla u^j \right) \nabla \varphi \\
&= -\frac{1}{\lambda^j} \int_A a^{\text{hom}} \nabla u^j \nabla T_{\eps} \varphi - \frac{1}{\lambda^j} \int_A \left[\chi_{\Omeps} \left( \nabla u^j + \nabla_y \chi_{\eps} \cdot \nabla u^j \right) - a^{\text{hom}} \nabla u^j \right] \nabla T_{\eps} \varphi,
\end{split}
\] 
now, applying lemma \ref{lemma periodic} on the periodic function \\$h\left(x/\eps\right) = \chi_{\Omeps} \left(1+ \nabla_y \chi_{\eps}  - a^{\text{hom}} \right)$, the continuity of the extension operator and equation (\ref{eq norme equiv}), one has
\[
\begin{split}
&\left| \frac{1}{\lambda^j} \int_{\Omeps \cap A}\left[\chi_{\Omeps} \left( \nabla u^j + \nabla_y \chi_{\eps} \cdot \nabla u^j \right) - a^{\text{hom}} \nabla u^j \right] \nabla T_\e \varphi\right| \\
&\leqslant \frac{C}{\lambda^j} \eps \|u^j \|_{H^1_0(A)} \|T_{\eps} \varphi \|_{H^1_0(\Omega)} \leqslant \frac{c_3}{\lambda^j} \eps^{1/2} \|\varphi \|_{\eps},
\end{split}
\]
so that
\[
-\frac{1}{\lambda^j} \int_{\Omeps} \nabla \Ujeps{j} \nabla \varphi 
\le -\frac{1}{\lambda^j} \int_A a^{\text{hom}} \nabla u^j \nabla T_{\eps} \varphi + \frac{c_2}{\lambda^j} \eps^{1/2}\|\varphi \|_{\eps}+ \frac{c_3}{\lambda^j} \eps^{1/2} \|\varphi \|_{\eps}.
\]

Therefore, putting together all these estimates, we get
\[
\begin{split}
 \|K_{\eps} \Ujeps{j} - \frac{1}{\lambda^j}\Ujeps{j} \|_{\eps} &\leqslant |Y| \int_A u^j T_{\eps} \varphi - \frac{1}{\lambda^j} \int_A a^{\text{hom}} \nabla u^j \nabla T_{\eps} \varphi \\
 &+ \eps^{1/2} c_1 \|\varphi \|_{\eps} + \eps^{1/2} \frac{c_2}{\lambda^j} \|\varphi \|_{\eps}  + \eps^{1/2} \frac{c_3}{\lambda^j} \|\varphi \|_{\eps}
\end{split}
\]

Note that, being $u^j$ the solution of the homogenized problem (\ref{Pb hom}), one has
\[
|Y| \int_A u^j T_{\eps} \varphi - \frac{1}{\lambda^j} \int_A a^{\text{hom}} \nabla u^j \nabla T_{\eps} \varphi =0,
\]
and, using equation (\ref{eq norme equiv}), we finally get
\[
 \|K_{\eps} \Ujeps{j} - \frac{1}{\lambda^j}\Ujeps{j} \|_{H_\eps}\leqslant \|K_{\eps} \Ujeps{j} - \frac{1}{\lambda^j}\Ujeps{j} \|_{\eps} \leqslant c^j \sqrt{\eps}
\]
\underline{\emph{Step 2.}}

To apply lemma \ref{lemma visik}, we need to use a normalized function: $\|\Ujeps{j}\|_{H_{\eps}}=1$. In our hypothesis we have $u^j$, $\nabla u^j \in C^{\infty}_0(A)$ and $\chi_{\eps} \in H^1_0(\Omeps)$, so that
\begin{equation}\label{Ujuj ineq}
\|\Ujeps{j} - u^j\|_{H_{\eps}} \leqslant C \eps
\end{equation}
and, being $u^j \neq 0$, we must have $\|\Ujeps{j}\|_{H_{\eps}} \geqslant\alpha> 0$. So we can use the normalized function, naming it again $U^j_{\eps}$:
\[
\Ujeps{j}=\frac{\Ujeps{j}}{\|\Ujeps{j}\|_{H_{\eps}}},
\]
getting
\[
 \|K_{\eps} \Ujeps{j} - \frac{1}{\lambda^j}\Ujeps{j} \|_{H_\eps} \leqslant c^j \sqrt{\eps} \frac{1}{\|\Ujeps{j}\|_{H_{\eps}}} \leqslant \frac{c^j}{\alpha} \sqrt{\eps}.
\]

Now we can apply lemma \ref{lemma visik} to the linear continuous compact and self-adjoint operator $K_{\eps}$, with $\mu=(\lambda^j)^{-1}$, $\alpha= c^j/\alpha \sqrt{\eps}$: then there exists $(\lajeps{j})^{-1}$, eigenvalue of $K_{\eps}$, such that
\[
\left|\frac{1}{\lajeps{j}} - \frac{1}{\lambda^j} \right| \leqslant \frac{c^j}{\alpha} \sqrt{\eps},
\]
moreover, for any $d>0$, there exists a normalized function $\tilde{u}_{\eps}$ in the eigenspace associated to eigenvalues in the interval $[\lajeps{j}-d,\lajeps{j}+d]$, such that
\[
\|\Ujeps{j} - \tilde{u}_{\eps} \|_{H_{\eps}} \leqslant 2 \frac{c^j \sqrt{\eps}}{\alpha d},
\]
that is equation (\ref{eq conv autof approx}) in an implicit form.
In order to understand better the convergence of eigenfunctions in the case of multiple eigenvalues, suppose to have $\lambda^j$ of multiplicity $m_j$, as in our hypothesis: 
\[
\lambda^{j-1}<\lambda^j=\lambda^{j+1}=\dots=\lambda^{j+m_j-1}<\lambda^{j+m_j},
\]
and set 
\[
d_j = \min \left( \frac{1}{\lambda^{j-1}} - \frac{1}{\lambda^j}, \frac{1}{\lambda^{j}} - \frac{1}{\lambda^{j+m_j}}\right)
\]
\[
\Lambda^j= \left( \frac{1}{\lambda^{j}} - d_j, \frac{1}{\lambda^{j}} + d_j\right),
\]
then $1/\lambda^i_{\eps} \in \Lambda^j$ if and only if $j \leqslant i \leqslant j+m_j-1$. For any of these $\lambda^i$ we construct the function $U^{j+i}_{\eps}(x)= u^{j+i}(x)+\eps \chi \left(\frac{x}{\eps} \right) \nabla u^{j+i}(x)$ and, repeating \emph{step 1}, we get
\[
 \|K_{\eps} \Ujeps{j+i} - \frac{1}{\lambda^{j+i}}\Ujeps{j+i} \|_{H_\eps}\leqslant \frac{c^{j+i}}{\alpha^{j+i}} \sqrt{\eps}, \quad j \leqslant i \leqslant j+m_j-1.
\]

Hence, by lemma \ref{lemma visik}, there exists an eigenfunction in the eigenspace associated to eigenvalues in the interval $\Lambda^j$, i.e. there exist a matrix $M_{\eps} \in \mathcal{M}^{m_j \times m_j}$ and an eigenfunction $\ueps^{j+i}$ associated to $\lambda^{j+i}_\eps$, with $1/\lambda^{j+i}_{\eps} \in \Lambda^j$, such that
\[
\|\Ujeps{j+i} - \sum_{l=0}^{m_j-1} M_{\eps}^{il} \ueps^{j+l} \|_{H_\eps} \leqslant 2  \frac{c^{j+i}}{\alpha^{j+i} d_j} \sqrt{\eps}=C^{j+i}\sqrt{\eps}, \quad j \leqslant i \leqslant j+m_j-1,
\]
that is equation (\ref{eq conv autof approx}); in order to derive equation (\ref{eq conv autof}) we simply note that, for any $j \in \mathbb{N}$, being $\chi \in H_{\eps}$ and $u^j \in \mathcal{C}^\infty_0(A)$,
\[
\|\Ujeps{j} -u^j \|_{H_{\eps}} = \eps \| \chi_{\eps} \nabla u^j \|_{H_{\eps}} \leqslant C \eps.
\]
\underline{\emph{Step 3.}}

We want to generalized to the case $u^j \in H^1_0(A)$; this means that $\nabla u^j$ could not be zero in $\partial A$, making $\Ujeps{j}$ not in $H^1_0(A)$ and inequality (\ref{Ujuj ineq}) holds true just in the $L^2(\Omega)$ norm. Consider $\psi_{\eps}$ a family of smooth functions in $C^{\infty}_0(A)$ such that $0 \leqslant \psi_{\eps} \leqslant 1$
\begin{equation}\label{def psieps}
\psi_{\eps}=	\begin{cases}\displaystyle
      			1	& \text{if}\  x \in A,\  d(x,\partial A) > 2\eps \\
      			0	& x \in \Omega \setminus A,
			\end{cases}
\end{equation}
and $\|\nabla \psi_{\eps}\|_{\infty} \leqslant 2/\eps$. Then take, for any $j \in \mathbb{N}$,

\[
\tilde{U}_{\eps}^j = u^j + \eps \psi_{\eps} \chi_{\eps} \nabla u^j,
\]
so that $\tilde{U}_{\eps}^j \in H^1_0(A)$.

The following estimates hold true
\[
\|\tilde{U}_{\eps}^j -\Ujeps{j} \|_{L^2(\Omega)} \leqslant C \eps^{3/2},
\]
\[
\|\tilde{U}_{\eps}^j -\Ujeps{j} \|_{H^1(\Omega)} \leqslant C \eps^{1/2}.
\]

Hence, repeating the proof of \emph{Step 1}, using $\tilde{U}_{\eps}^j$ instead of $\Ujeps{j}$, and these last estimates, we get the thesis in the general case.
\end{proof}

\bigskip

\end{document}